\documentclass[a4,psfig,epsf,11pt]{article}
\usepackage[dvips]{graphicx}
\usepackage[latin1]{inputenc}
\usepackage{amsfonts}
\usepackage{amsmath}
\usepackage{color}
\usepackage[mathcal]{euscript}
\usepackage{url}
\newtheorem{proposicao}{Proposition}[section]
\newtheorem{teorema}{Theorem}[section]
\newtheorem{lema}{Lemma}[section]
\newtheorem{corolario}{Corollary}[section]
\newtheorem{exem}{Example}[section]
\newtheorem{obs}{Remark}[section]
\newenvironment{proof}{{\noindent\bf Proof.} }
                       {\hfill\rule{2.1mm}{2.1mm} \bigskip }

\newtheorem{definicao}{Definition}[section]
\begin{document}
\title{Logarithmic Quasi-distance Proximal Point Scalarization Method for Multi-Objective Programming\thanks{This research was performed at the Computing and Systems Engineering Department, Rio de Janeiro Federal University, where the first author is conducting his doctoral studies, and was partially supported by CAPES.}}
\author{Rog\'erio Azevedo Rocha\thanks{Corresponding author. Tocantins Federal University, Undergraduate Computation Sciences Course, ALC NO 14 (109 Norte), C.P. 114, CEP 77001-090, Tel: +55 63 3232-8027 $\setminus$ +55 63 9221-5966 FAX +55 63 3232-8020,  Palmas, Brazil (azevedo@uft.edu.br).} \and{Paulo Roberto Oliveira\thanks{Rio de Janeiro Federal University, Computing and Systems Engineering Department, Caixa Postal 68511, CEP 21945-970, Rio de Janeiro, Brazil (poliveir@cos.ufrj.br).}} \\
\and{Ronaldo Greg\'orio\thanks{Rio de Janeiro Federal Rural University, Technology and Languages Department, Rua Capito Chaves, N 60, Centro, Nova Iguau, Rio de Janeiro, CEP 26221-010, Brazil (rgregor@ufrrj.br).}}}
\date{\today}
\maketitle
\begin{abstract}
Recently, Greg\'orio and Oliveira developed a proximal point scalarization method (applied to multi-objective optimization problems) for an abstract strict scalar representation with a variant of the logarithmic-quadratic function of Auslender et al. as regularization. In this study, a variation of this method is proposed, using the regularization with logarithm and quasi-distance, which entails losing important properties, such as convexity and differentiability.
However, proceeding differently, it is shown that any sequence $\{(x^k,z^k)\} \subset R^n\times R^m_{++}$ generated by the method satisfies: 
$\{z^k\}$ is convergent and $\{x^k\}$ is bounded and its accumulation points are weak pareto solutions of the unconstrained multi-objective optimization problem 
\end{abstract}
\section{Introduction}
\ \ \ \ This study considers the unconstrained multi-objective optimization problem
\\[-0.8cm]
\begin{eqnarray}
{\rm min} \ \{ F(x):\ \ x\in R^n \} \label{probmultiobjetivo}
\end{eqnarray}
where $F = (F_1,F_2,...,F_m)^T:R^n\rightarrow R^m$ is a convex mapping related to the lexicographic order generated by the cone $R^m_+$, i.e., for all $x,y\in R^n$ and $\lambda\in (0,1)$, 
$F_i(\lambda x + (1 - \lambda) y)\leq \lambda F_i(x) + (1 - \lambda)F_i(y), \ \ \forall i = 1,...,m$. Moreover, it will be required that one of the objective 
functions must be coercive, i.e., there is $r\in \{ 1,...,m\}$ such that $\lim_{\|x\|\rightarrow\infty} F_r(x) = \infty$.

The importance of multi-objective optimization can be seen from the large variety of applications presented in the literature. 
White \cite{White} offers a bibliography of 504 papers describing various different applications addressing, 
for example, problems concerning agriculture, banking, health services, energy, industry and water. More information 
with regard to multi-objective optimization can be found, for example, in section \ref{mp} and Miettinen \cite{Kaisa}.

There is a more general class of problems, known as vector optimization, that contains multi-objective optimization. 
See, for example, Luc \cite{Luc}. On the other hand, the methods developed for this class of problem can be classified into 
two types: scalarization methods and extensions of nonlinear algorithms to vector optimization. Some global 
optimization techniques are discussed in Chinchuluun and Pardalos \cite{Chinchuluun2}.

The classic proximal point method to minimize a scalar convex function $f:R^n\rightarrow R $ generates a sequence
 $\{ x^k \} $ via the iterative scheme: given a starting point $x^0 \in R^n $, then
\begin{equation}
  x^{k+1} \in {\rm argmin} \{ f(x) + \lambda_k \| x - x^k \|^2 : \ \ x\in R^n\},
\label{Alg-original} 
\end{equation}
where $\lambda_k$ is a sequence of real positive numbers and $\|.\|$ is the usual norm. This method was originally 
introduced by Martinet \cite{Martinet} and developed and studied by Rockafellar \cite{Rockafellar 0}. 
In recent decades the convergence analysis of the sequence $\{x^k\}$ has been extensively studied, and several extensions of the method 
have been developed in order to consider cases in which the function $f$ is not convex and/or cases 
where the usual quadratic term in (\ref{Alg-original}) is replaced by a generalized distance, e.g., Bregman distances, 
$\varphi$-divergences, proximal distances and quasi-distances. 
The papers containing these generalizations include: Chen and Teboulle \cite{Chen1}, Iusem and Teboulle \cite{Iusem1}, Pennanen \cite{Pennanen}, 
Hamdi \cite{abdelouahed}, Chen and Pan \cite{Chen2}, Papa Quiroz and Oliveira \cite{Papa}, Moreno et al. \cite{Moreno} and 
Langenberg and Tichatschke \cite{Langenberg}.

This class of proximal point algorithms has been extended to vector optimization. The first method in this direction was the multi-objective 
proximal bundle method (see, Mietttinen \cite{Kaisa}). G\"opfert et al. \cite{Gopfert} have presented a proximal point method for the scalar representation 
$\langle F(x), z\rangle$ with a regularization based on Bregman functions on finite dimensional spaces. Bonnel et al. \cite{Iusem} and Ceng and Yao \cite{Ceng} 
present a proximal algorithm with a quadratic regularization in vector form. Villacorta and Oliveira \cite{Villacorta} also present a proximal 
algorithm in vector form with the regularization being a proximal distance. Greg\' orio and Oliveira \cite{Gregorio} present a proximal 
algorithm in multi-objective optimization for an abstract strict scalar representation with a variant of the logarithmic-quadratic 
functions of Auslender et al. \cite{Auslender} as regularization. 

We will present a brief description of the method of Greg\' orio and Oliveira \cite{Gregorio}. 
Let $F: R^n\rightarrow R^m$ be a convex application. Given the starting points $x^0\in R^n $ and $z^0\in R^m_{++}$ and 
sequences $\beta_k,\mu_k  > 0,  k = 0,1,...$, the method generates a sequence $\{(x^k,z^k)\}\subset R^n \times R^m_{++}$ via the iterative scheme:
{\small 
\begin{equation}
(x^{k+1}, z^{k+1}) \in {\rm argmin} \ \{ f(x,z) + \beta_k H(z) + \frac{\alpha_k}{2} \| x - x^k \|^2 :  \ x\in \Omega^k, z\in R^n_{++}\} 
\label{metodoGregorioePaulo}
\end{equation}}where {\small $\Omega^k = \{ x \in R^n :  F_i(x) \leq F_i(x^k)$\}}, $f: R^n \times R^m_+ \rightarrow R$ satisfies the properties (P1) to (P4) 
(see section \ref{metodo}) and $H: R^m_{++} \rightarrow R$ is such that $H(z) = \langle z/z^k - {\rm log}( z/z^k) - e, e\rangle$
where  $e = (1,...,1)\in R^m$, $z/z^k = (z_1/z^k_1,...,z_m/z^k_m)$ and ${\rm log}(z/z^k) = ({\rm log} (z_1/z^k_1),...,{\rm log} (z_m/z^k_m))$. 

We are proposing a generalization of this method considering (\ref{metodoGregorioePaulo}) with the
quasi-distance $q: R^n \times R^n \rightarrow R_+$ (see definition \ref{defQD}) in place of the 
Euclidean norm $\| . \|$, i.e.,
{\small 
\begin{equation}
(x^{k+1}, z^{k+1}) \in {\rm argmin} \ \{ f(x,z) + \beta_k H(z) + \frac{\alpha_k}{2} q^2(x, x^k):  \ x\in \Omega^k, z\in R^n_{++}\}. 
\label{metodoR,O&G}
\end{equation}}As quasi-distances are not necessarily symmetric (see definition \ref{defQD}), they generalize the distances. 
Therefore, our algorithm generalizes Greg\' orio and
Oliveira's algorithm \cite{Gregorio}. A quasi-distance is not necessarily a convex function,
nor continuously differentiable, nor even a coercive function in any of its
arguments. Supposing that the quasi-distance satisfies the condition (\ref{prop-q}) (see
section \ref{qd}), then the coercivity and Lipschitz properties are recovered (see
Propositions \ref{prop-qd 1} and \ref{prop-qd 2}). However, important properties such as the
convexity and differentiability will be lost. Accordingly, we had to proceed differently to guarantee
the convergence of our method. What is more, we found a new example of a function $f: R^n \times R^m_+ \rightarrow R$ which 
satisfies the properties (P1) to (P4) (see proposition \ref{exemplo p5}), which were fundamental to the convergence of our method. 
Greg\' orio and Oliveira \cite{Gregorio}, drawing on the work of Fliege and Svaiter \cite{Fliege}, supposed that the set $\Omega^0$ is
limited and established the convergence of their method. In our case, a condition of coercivity was imposed on
us in only one of the objective functions, i.e., suppose
that there is  $r\in \{ 1,...,m\}$ such that $\lim_{\|x\|\rightarrow\infty} F_r(x) = \infty$, and that it has as a consequence the limitation of $\Omega^0$ (see Lemma \ref{Compacto}). The
importance of the limitation of the set $\Omega^0$ is that it guarantees that the sequence $\{x^k\}$ generated by our algorithm is limited (see proposition proof \ref{prop sequencias} (i)).

Quasi-distances can be applied not only to computer theory (see, for example,
Brattka \cite{Brattka} and Kunzi et al. \cite{Kunzi}), but also to economy, for
example and, more directly, to consumer choice and to utility functions (see,
for example, Romanguera and Sanchis \cite{Romaguera} and Moreno et al. \cite{Moreno}). Note
that Moreno et al. \cite{Moreno} developed a proximal algorithm with quasi-distance
regularization applied to non-convex and non-differentiable scalar functions, satisfying
the Kurdyka-Lojasiewics inequality. And because the quasi-distance is not necessarily
symmetric, they derived an economic interpretation of this algorithm, applied to 
habit formation. In this respect, the work of Moreno et al. encourages us, in future
investigations, to seek an economic interpretation of our algorithm applied to economy-related multi-objective
problems.

One important point is that our proximal algorithm and the proximal algorithm
developed by Greg\' orio and Oliveira \cite{Gregorio} were developed in multiobjective optimization and belong to the class of 
proximal point scalarization methods. Meanwhile, the algorithms developed by Bonnel et al. \cite{Iusem}, Ceng and Yao \cite{Ceng}, and 
Villacorta and Oliveira \cite{Villacorta} were developed in vector optimization and belong to the class of proximal methods in vector form. This means that the subproblems of our algorithm and Greg\' orio and Oliveira's are problems 
relating to the minimization of scalar functions, while the subproblems of the other studies' algorithms are problems relating to the minimization of vector functions.

Section \ref{qdst} presents concepts and results relating to quasi-distance and subdifferential theory.
In section \ref{mp}, concepts and results of general multi-objective optimization theory are presented. Section \ref{metodo} presents the authors' own method, where we assure the existence of the iterations, stop criterion and convergence. In section \ref{caso-q}, a variation of that method is considered. 
Finally, in section \ref{matlab}, the method is tested and numerical examples are offered.

\section{Quasi-distance and Subdifferential Theory}
\label{qdst}
In this section, the quasi-distance application is defined, with examples, and some of its properties that are fundamental to the course of our work are presented. The concepts of Fr\'echet subdifferential and limiting subdifferential are also revisited, along with some of their properties.
\subsection{Quasi-Distance}
\label{qd}
\begin{definicao}
[\cite{Stojmirovic}] Let $X$ be a set. A mapping $q: X \times X \rightarrow R_+$ is called a quasi-distance if for all $ x,y,z \in X,$ 
\begin{enumerate}
\item $q(x,y) = q(y,x) = 0 \Longleftrightarrow x = y$  
\item $q(x,z) \leq q(x,y) + q(y,z). \ \ \ \ \ \ \  $ 
\end{enumerate}
\label{defQD}
\end{definicao}
Notice that if $q$ satisfies the property of symmetry, i.e., if for all $x,y \in X, q(x,y)= q(y,x)$, then $q$ is a distance. A quasi-distance is not necessarily a convex function and coercive in the first argument, nor in the second (see \cite{Moreno}, Example 3.1 and Remark 3). Moreno et al. \cite{Moreno} presented the following example of quasi-distance.

\begin{exem}
 For each $i = 1,...,n,$, consider $c^-_i,c^+_i\ > 0$ and $q_i:R \times R \rightarrow R_+$ defined by
\begin{center}
$ q_i(x_i,y_i) =  \left\{
\begin{matrix}
c^+_i(y_i-x_i) & if &  y_i - x_i > 0 \\  c^-_i(x_i-y_i) & if &  y_i - x_i \leq 0 
\end{matrix}
\right.
$
\end{center}
is a quasi-distance on $R$, therefore $q(x,y) = \sum\limits_{i=1}^n q_i(x_i,y_i)$ is a quasi-distance on $R^n$. On the other hand, for each $\bar{z}\in R^n$,
\begin{center}
$q(x,\bar{z})= \sum\limits_{i=1}^n q_i(x_i,\bar{z}_i)= \sum\limits_{i=1}^n max \{c_i^+(\bar{z}_i - x_i),c_i^-(x_i - \bar{z}_i)\}, \ \ \ \ \ x\in R^n,$ 
\end{center}
thus $q(.,\bar{z})$ is a convex function. By the same reasoning, $q(\bar{z},.)$ is convex.
\label{exemplo economico}
\end{exem}
Moreno et al. \cite{Moreno} took into account the following condition with regard to the quasi-distance $q$: There are positive constants $\alpha$ and $\beta$
such that
\begin{eqnarray}
\alpha \| x - y \| \leq q(x,y) \leq \beta \| x - y \|, \ \ \ \ \forall x,y \in R^n  \label{prop-q}
\end{eqnarray}
Notice that the quasi-distance in the example \ref{exemplo economico} exhibits the property (\ref{prop-q}). With this condition, Moreno et al. \cite{Moreno} showed that, in each one of the arguments, the quasi-distance exhibits important properties, such as Lipschitz and coercivity. The results follow below:

\begin{proposicao}[\cite{Moreno}, Propositions 3.6 and 3.7]
Let $q: R^n \times R^n \rightarrow R_+$ be a quasi-distance that exhibits (\ref{prop-q}). Then for each $\bar{z} \in R^n$ the functions
$q(\bar{z},.)$ and $q(.,\bar{z})$ are Lipschitz continuous and the functions $q^2(\bar{z},.)$ and $q^2(.,\bar{z})$ are locally Lipschitz continuous functions on $R^n$.
\label{prop-qd 1}
\end{proposicao}
\begin{proposicao}[\cite{Moreno}, Remark 5]
Let $q: R^n \times R^n \rightarrow R_+$ be a quasi-distance that exhibits (\ref{prop-q}). Then for each $\bar{z} \in R^n$ the functions
$q(\bar{z},.)$, $q(.,\bar{z})$, $q^2(\bar{z},.)$ and $q^2(.,\bar{z})$ are coercive. 
\label{prop-qd 2} 
\end{proposicao}
\subsection{Subdifferential Theory}
Here Frechet's concepts of subdifferential and limiting subdifferential are recalled. Only the results fundamental to our study are presented. For more details, see \cite{Rockafellar}.
\begin{definicao}
Let $h: R^n \rightarrow R \cup \{ \infty \}$ be a proper lower semi-continuous function and $x \in R^n$. 
\begin{enumerate}
\item The Fr\'echet subdifferential of $h$ at $x,\hat{\partial} h(x)$, is defined as follows
\begin{center}
{\footnotesize $ \hat{\partial} h(x) := \left\{
\begin{matrix}
\left\{ x^* \in R^n : \liminf \limits_{y\neq x, y \rightarrow x} \dfrac{h(y)-h(x)- \langle x^*,y - x\rangle}{\lVert x - y\rVert} \geq 0 \right \} , & {\rm if} \  x \in dom(h) \\  
\O{}, \ \ \ \ \ \ \ \ \ \ \ \ \ \ \ \ \ \ \ \ \ \ \ \ \ \ \ \ \ \ \ \ \ \ \ \ \ \ \ \ \ \ \ \ \ \ \ \ \ \ \ \ \ \ \ \ \ \  & {\rm if} \  x \notin dom (h)
\end{matrix}
\right.$ }
\end{center}
\item  The limiting-subdifferential of $h$ at $x\in R^n, \partial h(x)$, is defined as follows
\begin{center}
$\partial h(x) := \left\{ x^* \in R^n : \exists x_n \rightarrow x, \ \ h(x_n) \rightarrow h(x), \ \ x^*_n \in \hat{\partial} h(x_n) \rightarrow x^* \right \} $  
\end{center}
\end{enumerate}
\end{definicao}
\begin{proposicao}[\textbf{Optimality condition} - \cite{Rockafellar}, Theorem 10.1] \ \\
If a proper function $h: R^n \rightarrow R \cup \{ + \infty \} $ has a local minimum at $\bar{x}$, then $0 \in \hat{\partial} h(\bar{x}), 0 \in \partial h(\bar{x}). $
\label{Prop-Cond.Otim.}
\end{proposicao}
\begin{obs}
 Let $C \subset R^n$. If a proper function $h: C \rightarrow R\cup \{\infty \}$ has a local minimum at $\bar{x}\in C$, then
$0 \in \hat{\partial}( h + \delta_C) (\bar{x}), 0 \in \partial( h + \delta_C)(\bar{x})$, where $\delta_C$ is the indicator function of the set $C$, defined as $\delta_C(x) = 0$ if $x$ belongs to $C$ and $\delta_C (x) = \infty$ otherwise.  
\label{indicadora}
\end{obs}
\begin{proposicao}[\cite{Rockafellar}, Exercise  10.10]
If $f_1$ is locally Lispschitz continuous at $\bar{x}$, $f_2$ is lower semi-continuous and proper with $f_2(\bar{x})$ finite, then
\begin{center}
 $\partial (f_1 + f_2)(\bar{x}) \subset \partial f_1 (\bar{x}) + \partial f_2 (\bar{x}) .$
\end{center}
\label{Prop-Sub.1}  
\end{proposicao}
\begin{proposicao}[\cite{Mordukhovich}, Theorem 7.1]
Let $f_i:R^n\rightarrow R, i = 1,2,$ be Lipschitz continuous around $\bar{x}$. If $f_i \geq 0, i = 1,2,$ then one has a product rule of the equality form 
\begin{center}
$\partial (f_1.f_2)(\bar{x}) = \partial (f_2 (\bar{x})f_1 + f_1 (\bar{x})f_2)(\bar{x}).$
\end{center}
\label{Prop.Sub.2}
\end{proposicao}
\begin{proposicao}[\cite{Rockafellar}, Proposition 5.15]
A mapping $S:R^n \rightarrow P (R^m)$ is locally bounded if and only if $S(B)$ is bounded for every bounded set $B$.
\label{prop Sub 1}
\end{proposicao}
\begin{proposicao}[\cite{Rockafellar}, Theorem 9.13]
Suppose $h: R^n \rightarrow R \cup \{ \pm \infty \}$ is locally lower semi-continuous at $\bar{x}$ with $h(\bar{x})$ finite.
Then the following conditions are equivalent:
\begin{enumerate}
\item [{\rm (a)}] $h$ is locally Lipschitz continuous at $\bar{x}$,
\item [{\rm (b)}] the mapping $\hat{\partial} h: x \mapsto \hat{\partial} h(x)$ is locally bounded at $\bar{x}$,
\item [{\rm (c)}] the mapping $\partial h: x \mapsto \partial h(x)$ is locally bounded at $\bar{x}$.
\end{enumerate}
Moreover, when these conditions hold, $\partial h(\bar{x})$ is non-empty and compact.
\label{prop Sub 2} 
\end{proposicao}


\section{Multi-objective programming - preliminary concepts}
\label{mp}
\paragraph{}
Only the concepts and results that are fundamental to the course of our work are presented. For more details, see, for example, Miettinen \cite{Kaisa} and Chinchuluun et al. \cite{Chinchuluun1}.

\begin{definicao}   
$a\in R^n$ is said to be a {\it\bf local pareto solution} to the problem (\ref{probmultiobjetivo}) if there is
a disc $B_{\delta}(a)\subset R^n$, with $\delta>0$, such that there is no $x\in B_{\delta}(a)$ that satisfies  $F_i(x)\leq F_i(a)$ for all $i = 1,...,m$ and $F_j(x) <  F_j(a)$ for at least one index $j\in\{1,...,m\}$.
\end{definicao} 

\begin{definicao}
$a\in R^n$ is known as a {\it\bf weak local pareto solution} if there is a disc $B_{\delta}(a)\subset R^n$, with $\delta>0$, such that there is no $x\in B_{\delta}(a)$ that satisfies  $F_i(x) < F_i(a)$ for all $i = 1,...,m$.
\end{definicao}

Generally, if a constrained or unconstrained multi-objective optimization problem is a convex problem, that is, if an objective function $F:R^n\rightarrow R^m$ is a convex function, then any (weak) local pareto solution is also a (weak) global pareto solution. This result is discussed in Theorem 2.2.3, in Miettinen \cite{Kaisa}. 

Let $\textnormal{argmin} \{F(x)|x \in R^n\}$ and $\textnormal{argmin}_w\{F(x)|x\in R^n\}$ denote the local pareto solution set and the local weak pareto solution set to the problem (\ref{probmultiobjetivo}). It is easy to see that
$\textnormal{argmin} \{F(x)|x\in R^n\}\subset \textnormal{argmin}_w \{F(x)|x\in
R^n\}$. 

\begin{definicao}
A real-valued function
$f:R^n\longrightarrow R$ is said to be a {\bf strict scalar representation} of
a map $F = (F_1,...,F_m):R^n\longrightarrow R^m$ when given $x,\bar{x}\in R^n$ 
\begin{equation*}
F_i(x)\leq F_i(\bar{x}), \ \forall \ i = 1,...,m \Longrightarrow f(x)\leq f(\bar{x})
\end{equation*}
 and 
\begin{equation*}
 F_i(x) < F_i(\bar{x}), \ \forall \ i = 1,...,m \Longrightarrow f(x)<f(\bar{x}).
\end{equation*}
Futhermore, $f$ is said to be a {\bf weak scalar representation} of $F$ if
\begin{equation*}
 F_i(x) < F_i(\bar{x}), \ \forall \ i = 1,...,m \Longrightarrow f(x)<f(\bar{x}).
\end{equation*} 
\label{escalarizacao1} 
\end{definicao} 
\vspace{-0,7cm}

It is obvious that all strict scalar representations are weak scalar representations. The next result demonstrates an interesting way to obtain scalar representation for applications. To demonstrate, see \cite{Gregorio}, proposition 1.
\begin{proposicao}
Let $f:R^n\longrightarrow R$ be a function. $f$ is a strict scalar representation of $F$ if, and only if $f$ is a composition of $F$ with a strictly increasing function $g:F\left(R^n\right)\longrightarrow R$.
\label{repconvexa}
\end{proposicao}
According to Proposition 2.9 in Luc \cite{Luc}, to obtain the convexity of the scalar problem, it is necessary to choose a convex increasing function $g$ from $F(R^n)$ to $R$. The next result establishes an important relation between the sets $\textnormal{argmin}\left\{f(x)|x\in R^n\right\}$ and $\textnormal{argmin}_w \{F(x)|x\in R^n\}$. The Proof follows immediately from the Definition \ref{escalarizacao1}.   
\begin{proposicao}
Let $f:R^n\longrightarrow R$ be a weak scalar representation of a map $F:R^n\longrightarrow R^m$ and $\textnormal{argmin}\left\{f(x)|x\in R^n\right\}$ the local
 minimizer set of $f$. Therefore:
\begin{equation*}
\textnormal{argmin}\left\{f(x)|x\in R^n\right\}\subset \textnormal{argmin}_w \{F(x)|x\in R^n\}.
\end{equation*} \label{escalarizacao2}
\end{proposicao} 

\section{Logarithmic Quasi-Distance Proximal point Scalarization $(LQDPS)$ Method}
\label{metodo}
Greg\'orio and Oliveira \cite{Gregorio} showed the existence of a function $f:R^n\times R^m_+\longrightarrow R$ that exhibits the following properties:
{\small
\begin{itemize}
\item[(P1)] $f$ is bounded below for any $\alpha\in R$, i.e, $f(x,z)\geq \alpha$ for every $(x,z)\in R^n\times R^m_+$; 
\item[(P2)] $f$ is convex in $R^n\times R^m_+$, i.e., given $(x_1,z_1), (x_2,z_2)\in R^n\times R^m_+$ and $\lambda\in (0,1)$
\begin{equation*}
f(\lambda(x_1,z_1)+(1-\lambda)(x_2,z_2))\leq \lambda f(x_1,z_1)+(1-\lambda)f(x_2,z_2);
\end{equation*}
\item[(P3)] $f$ is a strict scalar representation of $F$, with respect to $x$, i.e., 
\begin{equation*}
F_i(x)\leq F_i(y) \ \forall i = 1,...,m \ \Rightarrow f(x,z)\leq f(y,z)
\end{equation*}
and
\begin{equation*}
F_i(x) < F_i(y) \ \forall i = 1,...,m \ \Rightarrow f(x,z)<f(y,z)
\end{equation*}
for every $x,y\in R^n$ and $z\in R^m_+$;
\item[(P4)] $f$ is differentiable, with respect to $z$ and
\begin{equation*}
\frac{\partial}{\partial z}f(x,z)=h(x,z),
\end{equation*}
where $h(x,z)=(h_1(x,z), \cdots, h_m(x,z))^T$ is a continuous map from $R^n\times R^m$ to $R^m_+$, i.e, $h_i(x,z)\geq 0$ for all $i=1,\cdots, m$.
\end{itemize}} 

More precisely, they showed that the function $f:R^n\times R^m_+\longrightarrow R$ such that  
\begin{equation}
f(x,z) = \sum\limits_{i=1}^m {\rm exp}(z_i + F_i(x)) \label{exemplo Gregorio e Oliveira}
\end{equation}
satisfies the properties $(P1)$ to $(P4)$. As another example, consider the following proposition:
\begin{proposicao}
 Let $F = (F_1,F_2,...,F_m): R^n \rightarrow R^m$ be a convex application, then $f:R^n \times R^m_+ \rightarrow R$ 
such that $f(x,z)=\sum\limits_{i=1}^m \left[ z_i + h(F_i(x))\right]$ where 
$ h(F_i(x)) =  \left\{
\begin{matrix}
\frac{1}{2 - F_i(x)} & if &  F_i(x) \leq 1 \\  (F_i(x))^2 & if &  F_i(x) > 1 
\end{matrix}
\right.
$
satisfies the properties (P1) to (P4).
\label{exemplo p5}
\end{proposicao}
\begin{proof} 
It is easy to see that $h: R\rightarrow R$ given by $ h(x) =  \left\{
\begin{matrix}
\frac{1}{2 - x} & if &  x \leq 1 \\  x^2 & if &  x > 1 
\end{matrix}
\right.
$ 
is positive ($h > 0$), convex and strictly increasing.
It will now be shown that $f(x,z)=\sum\limits_{i=1}^m \left[ z_i + h(F_i(x))\right]$ satisfies the properties (P1) to (P4). 
{\bf (P1):}  $h(x) > 0 \  \forall x\in R$ and $z_i\geq 0  \ \forall i=1,...,m$ implies $f(x,z) > 0 \ \forall (x,z)\in R^n\times R_+^m$.
{\bf (P2):} Be $(x,z),(\bar{x},\bar{z})\in R^n\times R^m_+$ and $\alpha\in [0,1]$. Then, as $F_i$ is convex $\forall i = 1,...,m$ and $h$ is
strictly increasing and convex: {\small
\begin{eqnarray*}
 f(\alpha(x,z) + (1 - \alpha)(\bar{x},\bar{z})) &  =   & \sum\limits_{i=1}^m \left[ \alpha z_i + (1 - \alpha)\bar{z_i} + h \left( F_i(\alpha x + (1 - \alpha)\bar{x}) \right) \right] \nonumber \\
                                                & \leq & \sum\limits_{i=1}^m \left[ \alpha z_i + (1 - \alpha)\bar{z_i} + h \left( \alpha F_i(x) + (1 - \alpha) F_i(\bar{x}) \right) \right] \nonumber \\
                                                & \leq & \sum\limits_{i=1}^m \left[ \alpha z_i + (1 - \alpha)\bar{z_i} + \alpha h \left(F_i(x)\right) + (1 - \alpha) h\left( F_i(\bar{x})\right) \right] \nonumber \\
                                                &  =   & \alpha \sum\limits_{i=1}^m \left( z_i + h \left(F_i(x)\right) \right) + (1 - \alpha)\sum\limits_{i=1}^m \left( \bar{z_i} + h\left( F_i(\bar{x})\right)\right) \nonumber \\
                                                &  =   & \alpha f(x,z) + (1 - \alpha)f(\bar{x},\bar{z}) \nonumber.
\end{eqnarray*}}
{\bf (P3):} Consider $\bar{z}\in R^m_+$ fixed. As $h$ is strictly increasing, $F_i(x)\leq F_i(y) \ \forall i = 1,...,m$ implies $\bar{z_i} + h(F_i(x))\leq \bar{z_i} + h(F_i(y)), \ \forall i = 1,...,m$ and
therefore, $\sum\limits_{i=1}^m \left[ \bar{z_i} + h(F_i(x))\right] \leq \sum\limits_{i=1}^m \left[\bar{z_i} + h(F_i(y))\right]$, i.e., $f(x,\bar{z})\leq f(y,\bar{z})$. 
Similarly $F_i(x) < F_i(y) \ \forall i = 1,...,m$ implies $f(x,\bar{z}) < f(y,\bar{z})$. 
{\bf (P4)} It can be easily shown that $\frac{\partial}{\partial z}f(x,z)=(1,1,...,1).$
\end{proof} 
\ \\ 
{\bf Notation:} \ \ Let $y,\bar{y} \in R^m$, then $y \leq \bar{y} \Longleftrightarrow y_i\leq\bar{y}_i \ \forall i = 1,...,m$ and
$y\ll\bar{y}\Longleftrightarrow y_i < \bar{y}_i \ \forall i = 1,...,m$. \\
\ \\ 
{\bf The $(LQDPS)$ Method:}\\
Let $F:R^n\longrightarrow R^m$ be convex and $q: R^n \times R^n \rightarrow R_+$ a quasi-distance application, satisfying (\ref{prop-q}). Given the initial points $x^0\in R^n$, $z^0\in R^m_{++}$ and sequences $\beta^k > 0, k=0,1,\cdots$ and 
$0 < l < \mu^k < L \ \forall k = 1,2,...$,
the logarithmic quasi-distance proximal point scalarization $(LQDPS)$ method generates sequences $\left\{x^k\right\}_{k\in N}\subset R^n$  and 
$\left\{z^k\right\}_{k\in N}\subset R^m_{++}$ with the iterates $x^{k+1}$ and $z^{k+1}$ defined as the solution of the $(LQDPS)$ problem\\
\begin{equation}
{\rm min} \ \varphi^k(x,z)=f(x,z)+\beta^k\displaystyle\langle\frac{z}{z^k}-{\rm log}
\  \frac{z}{z^k}-e,e\displaystyle\rangle+\frac{\mu^k}{2}q^2(x,x^k),\label{exponencial}
\end{equation}
\begin{equation*}
x\in \Omega^k, z\in R^m_{++},
\end{equation*}
where $f:R^n\times R^m_+\longrightarrow R$ satisfies the properties (P1) to (P4), $\displaystyle\frac{z}{z^k}$ 
and ${\rm log}\displaystyle \frac{z}{z^k}$ which are the vectors whose $ith$ components are given 
by $\displaystyle\frac{z_i}{z^k_i}$ and ${\rm log}\displaystyle\frac{z_i}{z^k_i}$, respectively,  
$e\in R^m$ is the vector with all components equal to 1 and $\Omega^k=\{x\in R^n|F(x)\leq F(x^k)\}$. 
\subsection{Well-posedness}
The function $\varphi^k:R^n\times R^m_{++}\longrightarrow R$ in (\ref{exponencial}), was considered by Greg\'orio and Oliveira \cite{Gregorio} with a 
variant of the logarithm-quadratic function of Auslender et al. \cite{Auslender} as regularization 
and, in that case, due to the strict convexity of the function $\varphi^k$, they showed that the method's iterations are unique and 
within the constraints. As, in the present problem, the quasi-distance is not necessarily a convex function, we will not be able to demonstrate the uniqueness of the iterations, nor that the iterations $x^{k+1}$ are within the constraints $\Omega^k$. 
Therefore, we will have to act differently to assure that the sequences are well-defined and also to find their characterizations.
\begin{lema}
Let $F:R^n\longrightarrow R^m$ be a convex map such that there exists $r\in \{ 1,...,m \}$ satisfying 
$\lim_{\|x\|\rightarrow\infty} F_r(x) = \infty$. Then $\Omega^k, \ \forall k\in N$ is a convex and compact set. Particularly, $\Omega^k\times R^m_+$ is a convex and closed set. 
\label{Compacto}
\label{prop omega}
\end{lema}
\begin{proof}
Suppose, for contradiction that $\Omega^0=\{x\in R^n|F(x)\leq F(x^0)\}$ is unbounded. Then there is
$\{x_n\}_{n\in N} \subset \Omega^0$ such that $\|x_n\| \rightarrow \infty$ when $n\rightarrow \infty$. As $\{x_n\}_{n\in N} \subset \Omega^0$, then $F(x_n)\leq F(x^0) \ \forall n\in N$, and then, $F_i(x_n)\leq F_i(x^0), \ \forall \ i = 1,...,m$ and $n\in N$. Therefore, in particular, $F_r(x_n)\leq F_r(x^0) \ \forall n\in N$. Since $F_r$ is coercive and $\|x_n\| \rightarrow \infty$ when $n\rightarrow \infty$, then $``\infty \leq F_r(x^0) < \infty''$, which is a contradiction. So $\Omega^0$ is limited. 
As $\Omega^{k+1} \subseteq \Omega^{k},\ k\geq0$, it follows that $\Omega^k \subseteq \Omega^0,\ k\geq1$ and therefore $\Omega^k$ is limited $\forall k\geq 0$. The convexity of $F$ implies its continuity and the convexity of $\Omega^k, \ \forall k$. It follows from the continuity of $F$ that $\Omega^k, \ \forall k$ is closed. Therefore, $\Omega^k \ \forall k$ is a compact convex set.
\end{proof}
\begin{obs}
Since $\Omega^{k+1}\subseteq \Omega^k, \forall k\in N$ and $\Omega^k, \forall k\in N$ is a compact set, then:  
$\bigcap\limits_{k=0}^{\infty} \Omega^k \neq \emptyset$.
\label{inters omegas} 
\end{obs}
\begin{lema}
The function $H:R^m_{++}\rightarrow R$, such that
\begin{center}
$H(z)= \left< \frac{z}{z^k}-{\rm log}\frac{z}{z^k} - e, e \right> = \displaystyle\|\frac{z}{z^k}-{\rm log} \ \frac{z}{z^k}-e\displaystyle\|_1$ 
\end{center}
where $\|\bullet\|_1$ is the 1-norm on $R^m$ defined by $\|z\|_1=\displaystyle\sum^{m}_{i=1}|z_i|$, is strictly convex, non-negative and coercive.
\label{Logaritmo}
\end{lema}
\begin{proof}
See \cite{Gregorio}, demonstration of Lemma 1.
\end{proof}
\begin{proposicao}[Well-posedness]
Let $F:R^n\longrightarrow R^m$ be a convex map such that there exists $r\in \{ 1,...,m \}$ satisfying $\lim_{\|x\|\rightarrow\infty} F_r(x) = \infty$, $q: R^n \times R^n \rightarrow R_+$ a quasi-distance map satisfying (\ref{prop-q}) and $f:R^n\times R^m_+\longrightarrow R$ be a function satisfying the properties (P1) to (P4). Then, for every $k\in N$, there is one solution $\left(x^{k+1},z^{k+1}\right)$ for the $(LQDPS)$ problem.
\label{boa definicao} 
\end{proposicao}
\begin{proof}
The function $\varphi^k : \Omega^k\times R^m_{++} \rightarrow R$  is coercive. In fact, for (P1) we have: 
\begin{eqnarray}
\varphi^k(x,z) &  =    & f(x,z) + \beta^k \left< \frac{z}{z^k}-{\rm log}\frac{z}{z^k} - e, e \right> + \frac{\mu^k}{2}q^2(x,x^k) \nonumber \\ 
               & \geq  & \alpha + \beta^k (\displaystyle\|\frac{z}{z^k}-{\rm log} \ \frac{z}{z^k}-e\displaystyle\|_1) + \frac{\mu^k}{2}q^2(x,x^k). \label{coerciva}
\end{eqnarray}
Let us define $\|(x,z)\| = \|x\| + \|z\|$ and suppose that $\|(x,z)\|\rightarrow\infty$. This is the same as $ \| x \| \rightarrow\infty$ or $\|z\|\rightarrow\infty$. As $\Omega^k$ is compact (see lemma \ref{prop omega}) and the function $ \displaystyle\|\frac{z}{z^k}-{\rm log} \ \frac{z}{z^k}-e\displaystyle\|_1$ is coercive in $R^m_{++}$ (see lemma \ref{Logaritmo}), it follows from (\ref{coerciva}) that $\varphi^k$ is coercive in $\Omega^k\times R^m_{++}.$\\
The function $\varphi^k : R^n\times R^m_{++} \rightarrow R$ is continuous in $R^n\times R^m_{++}$. In fact, (P2) implies $f$ is continuous in $R^n \times R^m_{++}$. The lemma \ref{Logaritmo} implies $H(z)=\left< \frac{z}{z^k}-{\rm log}\frac{z}{z^k} - e, e \right>$ continuous in $R^m_{++}$. As a consequence of proposition \ref{prop-qd 1}, $q^2(.,x^k):R^n\rightarrow R$ is a continuous application in $R^n$. Threfore, the function $\varphi^k : R^n\times R^m_{++} \rightarrow R \cup \{ + \infty \} $ is continuous in $R^n\times R^m_{++}$.\\
As $\varphi^k : \Omega^k\times R^m_{++} \rightarrow R$ is continuous, coercive and proper in $\Omega^k\times R^m_{++}$, we have that the set argmin$\left\{\varphi^k(x,z) / (x,z)\in\Omega^k\times R^m_{++}\right\}$ is not empty, i.e., for every $k$, there is a solution 
$\left(x^{k+1},z^{k+1}\right)$ to the $(LQDPS)$ problem. 
\end{proof}
\begin{definicao}
Let $C\subset R^n$ be a convex set and $\bar{x}\in C$. The {\it normal cone} (cone of normal directions) at the point $\bar{x}$ related to the set $C$ is given by
\begin{center}
 $N_C(\bar{x})=\{ v \in R^n \ \ / \ \ \left< v , x - \bar{x} \right>   \leq 0 \ \ \forall x\in C \}.$
\end{center}
\label{cone}
\end{definicao}
\begin{corolario}[Characterization] 
\ \\ The solutions $\left(x^{k+1},z^{k+1}\right)$ of the $(LQDPS)$ problem are characterized by:
\begin{enumerate}
\item[{\rm (i)}] {\rm There are } $\xi^{k+1}\in\partial f(.,z^{k+1})(x^{k+1})$, \ \ \ $\zeta^{k+1}\in\partial(q(.,x^k))(x^{k+1})$ \ \ {\rm and} \ \ \\ $v^{k+1}\in N_{\Omega^k}(x^{k+1})$ {\rm such that}
\begin{equation}
 \xi^{k+1} = - \mu^kq(x^{k+1},x^k)\zeta^{k+1} - v^{k+1} \label{caract. em x}
\end{equation}
and
\item[{\rm (ii)}]
\begin{equation}
\frac{1}{z^{k+1}_i}-\frac{1}{z^k_i}=\frac{h_i\left(x^{k+1},z^{k+1}\right)}{\beta^k},
 \ \ \ i=1,\cdots m. \label{caract. em z} 
\end{equation}
\begin{equation*}
x^{k+1}\in \Omega^k, z^{k+1}\in R^m_{++}
\end{equation*}
\end{enumerate}
\label{Caract}
\end{corolario}
\begin{proof}

From observation \ref{indicadora} we have
{\footnotesize 
\begin{equation}
0 \in \partial \left( f(.,z^{k+1}) + \beta^k\left < \frac{z^{k+1}}{z^k}-{\rm log}\frac{z^{k+1}}{z^k}-e,e\right> + \frac{\mu^k}{2}q^2(.,x^k) + \delta_{\Omega^k}\right)(x^{k+1}).
\label{otim 1}
\end{equation}} From (P2), $f(.,z^{k+1}) + \beta\left < \frac{z^{k+1}}{z^k}-{\rm log}\frac{z^{k+1}}{z^k}-e,e\right>$
is continuous in $x^{k+1}$; from proposition \ref{prop-qd 1}, $\frac{\mu^k}{2}q^2(.,x^k)$ is locally Lipschitz in $x^{k+1}$;
the convexity of $\Omega^k$ implies the convexity of $\delta_{\Omega^k}$ and therefore that $\delta_{\Omega^k}$ is locally Lipschitz. Therefore, using the proposition \ref{Prop-Sub.1} in (\ref{otim 1}) and remembering that 
\begin{equation*}
\beta\left < \frac{z^{k+1}}{z^k}-{\rm log}\frac{z^{k+1}}{z^k}-e,e\right>  
\end{equation*}
is constant in relation to $\Omega^k$, we obtain
\begin{equation}
0\in \partial \left( f(.,z^{k+1})\right)(x^{k+1}) + \partial \left(\frac{\mu^k}{2}q^2(.,x^k)\right)(x^{k+1}) + \partial \left(\delta_{\Omega^k}\right)(x^{k+1}).
\label{otim 2}
\end{equation}
As $\Omega^k$ is closed and convex, it follows $\partial \left(\delta_{\Omega^k}(.)\right)(x^{k+1}) = N_{\Omega^k}(x^{k+1})$, where $N_{\Omega^k}(x^{k+1})$ denotes the normal cone at the point $x^{k+1}$ in relation to the set $\Omega^k$ (see def. \ref{cone}). From the propositon \ref{prop-qd 1}, $q(.,x^k)$ is Lipschitz continuous in $R^n$. Therefore, taking $f_1 = f_2 = q$ in the proposition \ref{Prop.Sub.2}, we have from (\ref{otim 2}) that
\begin{center}
 $0\in \partial \left( f(.,z^{k+1})\right)(x^{k+1}) + \mu^kq(x^{k+1},x^k)\partial \left(q(.,x^k)\right)(x^{k+1}) + N_{\Omega^k}(x^{k+1})$, 
\end{center}
i.e., there are $\xi^{k+1}\in\partial f(.,z^{k+1})(x^{k+1})$, $\zeta^{k+1}\in \partial (q(.,x^k))(x^{k+1})$ and $v^{k+1} \in N_{\Omega^k}(x^{k+1})$ such that
\begin{equation*}
\xi^{k+1} = - \mu^kq(x^{k+1},x^k)\zeta^{k+1} - v^{k+1}.
\end{equation*}
To end the demonstration, observe (see \cite{Gregorio}, Lemma 1) that
\begin{equation*}
\frac{1}{z^{k+1}_i}-\frac{1}{z^k_i}=\frac{h_i\left(x^{k+1},z^{k+1}\right)}{\beta^k},
 \ \ \ i=1,\cdots m. 
\label{caract z}
\end{equation*}
\begin{equation*}
x^{k+1}\in \Omega^k, z^{k+1}\in R^m_{++}
\end{equation*}

\end{proof}
 
\subsection{STOP CRITERION}

As in Greg\'orio and Oliveira \cite{Gregorio}, we will establish the same stopping rule as was used by Bonnel et al in \cite{Iusem}.

\begin{proposicao}[Stop criterion]
Let $\{(x^k,z^k)\}_{k\in N}$ be the sequence generated by the $(LQDPS)$ method. If $(x^{k+1},z^{k+1})=(x^k,z^k)$ for any integer $k$ 
then $x^k$ is a weak pareto solution for the unconstrained multi-objective optimization problem (\ref{probmultiobjetivo}). 
\label{parada}
\end{proposicao}
\begin{proof}
Now, suppose that the stopping rule is satisfied in the $kth$ iteration. By contradiction, admit that $x^k$ is not a weak pareto solution. Then, there is $\bar{x}\in R^n$ such that $F(\bar{x})\ll F(x^k)$. From (P3) we have
\begin{equation*}
f(\bar{x},z^k)<f(x^k,z^k).
\end{equation*}
This implies that there exists $\alpha>0$ such that $f(\bar{x},z^k)=f(x^k,z^k)-\alpha$. Define $x_{\lambda}=\lambda x^k+(1-\lambda)\bar{x}, \lambda\in (0,1)$. Then we have that
\begin{equation*}
(x_\lambda,z^k)=\lambda(x^k,z^k)+(1-\lambda)(\bar{x},z^k).
\end{equation*}
Since $(x^{k + 1},z^{k + 1})$ solves the $(LQDPS)$ problem, $(x^{k+1},z^{k+1})=(x^k,z^k)$, $q^2(x^k,x^k) = 0$ and $x_\lambda\in\Omega^k ,\ \forall \lambda\in (0,1)$, we obtain,
\begin{center}
$ f(x^k,z^k)\leq f(x_\lambda,z^k) + \frac{\mu^k}{2}q^2(x_{\lambda},x^k), \ \forall \ \lambda\in (0,1)$. 
\end{center}
From (\ref{prop-q}), we have,
\begin{equation}
 f(x^k,z^k)\leq f(x_\lambda,z^k) + \frac{\mu^k}{2}\beta^2\|x_{\lambda}-x^k\|^2, \ \forall \ \lambda\in (0,1).
\label{Parada 1}
\end{equation}
As $x_{\lambda} - x^k = (1 - \lambda)(\bar{x} - x^k)$, of (\ref{Parada 1}), then
\begin{equation}
 f(x^k,z^k)\leq f(x_\lambda,z^k) + \frac{\mu^k}{2}\beta^2(1 - \lambda)^2\|\bar{x} - x^k\|^2, \ \forall \ \lambda\in (0,1).
\label{Parada 2}
\end{equation}
On the other hand, the convexity of $f$ implies that
\begin{eqnarray}
 f(x_{\lambda},z^k) & \leq & \lambda f(x^k,z^k) + (1 - \lambda) f(\bar{x},z^k) \nonumber \\
                    & = & \lambda f(x^k,z^k) + (1 - \lambda)(f(x^k,z^k) - \alpha) \nonumber \\
                    & = & f(x^k,z^k) - (1 - \lambda)\alpha. \label{Parada 3} 
\end{eqnarray}
From (\ref{Parada 2}) and (\ref{Parada 3}), $f(x^k,z^k)\leq f(x^k,z^k) - (1 - \lambda)\alpha + \frac{\mu^k}{2}\beta^2(1 - \lambda)^2\|\bar{x} - x^k\|^2 $. So
\begin{center}
 $\alpha \leq (1 - \lambda)\frac{\mu^k}{2}\beta^2\|\bar{x} - x^k\|^2, \ \forall \lambda \in (0,1).$
\end{center}
Hence, $\alpha \leq \displaystyle\lim_{\lambda\rightarrow 1^-}(1 - \lambda)\frac{\mu^k}{2}\beta^2\|\bar{x} - x^k\|^2,$ and therefore, $\alpha\leq 0$, which is a contradiction.
Therefore, $x^k$ is a weak pareto solution for the unconstrained multi-objective optimization problem (\ref{probmultiobjetivo}).
\end{proof}

\subsection{CONVERGENCE}
Based on Fliege and Svaiter \cite{Fliege}, Greg\'orio and Oliveira \cite{Gregorio}, assuming that $\Omega^0$ is limited, established the convergence 
of the log-quadratic proximal scalarization method. In this study, we assume that one of the objective functions is coercive. Consequently, $\Omega^0$ is limited (see the proof of lemma \ref{prop omega}).
\begin{proposicao}
 Let $\{(x^k,z^k)\}_{k\in N}$ be a sequence generated by the {\bf$(LQDPS)$ Method}. Then {\rm (i)} $\{x^k\}_{k\in N}$ is bounded;
  {\rm (ii)} $\{z^k\}_{k\in N}$ is convergent; {\rm (iii)} $\{f(x^k,z^k)\}_{k\in N}$ is noincreasing and convergent.
\label{prop sequencias}
\end{proposicao}
\begin{proof}
{\bf (i)} Since $\Omega^k\supseteq\Omega^{k+1}, k = 0,1,...,$, we have $x^k\in\Omega^{k-1}\subseteq\Omega^0 \ \ \forall k\geq 1$.
As $\Omega^0$ is limited, it follows that  $\{x^k\}$ is limited.\\
{\bf (ii)}  Since $h_i(x,z)\geq 0$, $\beta^k > 0$ and $\{z^k_i\}_{k\in N}$ is bounded below, Equation (\ref{caract. em z}) implies $\{z^k\}_{k\in N}$ is convergent (see \cite{Gregorio}, proof of theorem 1).\\
{\bf (iii)} $\varphi^k(x^{k+1},z^{k+1})\leq\varphi^k(x^k,z^k), \forall k \in N$, i.e, for every $k\in N$,
{\small
\begin{equation}
 f(x^{k+1},z^{k+1}) + \beta^k\left < \frac{z^{k+1}}{z^k}-{\rm log}\frac{z^{k+1}}{z^k}-e,e\right> + \frac{\mu^k}{2}q^2(x^{k+1},x^k)\leq f(x^k,z^k).
\label{conv}
\end{equation}}
As $\beta^k\left < \frac{z^{k+1}}{z^k}-{\rm log}\frac{z^{k+1}}{z^k}-e,e\right> + \frac{\mu^k}{2}q^2(x^{k+1},x^k) \geq 0 \ \forall k \in N$, 
we have,
\begin{center}
$f(x^{k+1},z^{k+1}) \leq f(x^k,z^k) \ \forall k \in N$, 
\end{center}
i.e., $\{f(x^k,z^k)\}_{k\in N}$ is a noincreasing sequence. For (P1), $\{f(x^k,z^k)\}$ is bounded below, and therefore convergent. 
\end{proof}
\begin{proposicao}
Let $\{x^k\}_{k\in N}$ be a sequence generated by {\bf $(LQDPS)$} Method. Then
\begin{enumerate}
 \item  [{\rm (i)}] $\displaystyle\sum_{k = 0}^{\infty}q^2(x^{k + 1},x^k) < \infty$. In particular $\displaystyle\lim_{k\rightarrow\infty}q^2(x^{k + 1},x^k) = 0$.
 \item [{\rm (ii)}] $\displaystyle\lim_{k\rightarrow\infty}\|x^k - x^{k+1}\| = 0$.
\end{enumerate}
\label{prop serie}
\end{proposicao}
\begin{proof}
 {\bf (i)} As $\beta^k\left < \frac{z^{k+1}}{z^k}-{\rm log}\frac{z^{k+1}}{z^k}-e,e\right> \geq 0$, of (\ref{conv}) we have:
\begin{center}
$f(x^{k+1},z^{k+1}) + \frac{\mu^k}{2}q^2(x^{k+1},x^k)\leq f(x^k,z^k), \ \forall k\in N.$  
\end{center}
\begin{eqnarray}
{\rm Hence,} \ \ \ \ \ \ \ \ \ \ \ \ \ \ \  q^2(x^{k+1},x^k)  & \leq & \frac{2}{\mu^k}\left( f(x^k,z^k) - f(x^{k+1},z^{k+1}) \right), \forall k\in N \nonumber \\
                                                             & \leq & \frac{2}{l}\left( f(x^k,z^k) - f(x^{k+1},z^{k+1}) \right), \forall k\in N. \nonumber 
\end{eqnarray}
Therefore, as long as $\{f(x^k,z^k)\}_{k\in N}$ is noincreasing and convergent, 
\begin{center}
$\displaystyle\sum_{k = 0}^{n}q^2(x^{k + 1},x^k) \leq \frac{2}{l}\left( f(x^0,z^0) - \displaystyle\lim_{k\rightarrow\infty}f(x^{k+1},z^{k+1}) \right) < \infty \ \ \forall n\in N$
\end{center}
{\bf (ii)} (\ref{prop-q}) implies $\alpha^2 \|x^k - x^{k+1}\|^2\leq q^2(x^{k + 1},x^k), \ \forall \ k\in N.$ So from {\bf (i)}, $\displaystyle\lim_{k\rightarrow\infty}\|x^k - x^{k+1}\| = 0$. 
\end{proof}
\begin{proposicao}
 If  $\{x^k\}_{k\in N}$ is bounded, then the set $\partial\left( q(.,x^k)\right)(x^{k+1})$ is bounded for every $k\in N.$
\label{Limitacao}
\end{proposicao}
\begin{proof}
 It follows from the propositions \ref{prop Sub 1} and \ref{prop Sub 2}, see \cite{Moreno} Lemma 5.1.
\end{proof}

Now, we can prove the convergence of our method if the stopping rule never applies.
\begin{teorema}[convergence]
Let $F:R^n\longrightarrow R^m$ be a convex map such that $\lim_{\|x\|\rightarrow\infty} F_r(x) = \infty$ for some $r\in \{ 1,...,m \}$, $f:R^n\times R^m_+\longrightarrow R$ be a function satisfying the properties (P1) to (P4), and $q: R^n\times R^n\rightarrow R_+$ be a quasi-distance function that satisfies (\ref{prop-q}) . If $\left\{\mu^k\right\}_{k\in N}$ and $\left\{\beta^k\right\}_{k\in N}$ are sequences of real positive numbers, with $ 0 < l < \mu^k < L, \forall \ k\in N$, then the sequence $\{(x^k,z^k)\}_{k\in N}$ generated by the logarithmic quasi-distance proximal point scalarization method is bounded and each cluster point of $\{x^k\}_{k\in N}$ is a weak pareto solution for the unconstrained multi-objective optimization
problem (\ref{probmultiobjetivo}). \label{convergencia}
\end{teorema}
\begin{proof}
From Proposition \ref{prop sequencias}, there are $x^*\in R^n$, $z^*\in R^m_+$ and $\{x^{k_j}\}_{j\in N}$, a subsequence of $\{x^k\}_{k\in N}$, such that $\displaystyle\lim_{j\rightarrow\infty}x^{k_j} =x^*$ and $\displaystyle\lim_{k\rightarrow\infty}z^{k} =z^*$. From (P2) and (P4) $f$ is continuous in $R^n\times R^m_+$, so $f(x^*,z^*) = \displaystyle\lim_{k\rightarrow\infty}f(x^{k_j},z^{k_j}) =  \displaystyle\inf_{k\in N}\{ f(x^k,z^k)\}$. 
From corollary \ref{Caract}(i), there are $\zeta^{k+1}\in\partial(q(.,x^k))(x^{k+1})$ and $v^{k+1}\in N_{\Omega^k}(x^{k+1})$ such that 
\begin{center}
$- \mu^kq(x^{k+1},x^k)\zeta^{k+1} - v^{k+1} \in \partial f(.,z^{k+1})(x^{k+1})$. 
\end{center}
Hence, from subgradient inequality for the convex function $f(.,z^{k+1})$ we have: $\forall x\in R^n$,
{\small
\begin{eqnarray}
 f(x,z^{k_j+1}) & \geq & f(x^{k_j + 1},z^{k_j + 1}) - \mu^{k_j}q(x^{k_j + 1},x^{k_j})<\zeta^{k_j + 1},x - x^{k_j + 1}> \nonumber \\
                &  -   & < v^{k_j+1},x - x^{k_j+1}> \label{des subgrad} 
\end{eqnarray}}As $v^{k_j+1}\in N_{\Omega^{k_j}}(x^{k_j+1})$ we have $ - < v^{k_j+1},x - x^{k_j+1}> \ \geq  0 \ \ \forall x\in\Omega^{k_j}$ (See definition \ref{cone}). 
By remark \ref{inters omegas}, $\Omega = \bigcap\limits_{k=0}^{\infty} \Omega^k \neq \emptyset$. Therefore, in particular, from (\ref{des subgrad}): $\forall x\in \Omega$,
{\small
\begin{equation}
f(x,z^{k_j+1}) \geq  f(x^{k_j + 1},z^{k_j + 1}) - \mu^{k_j}q(x^{k_j + 1},x^{k_j})<\zeta^{k_j + 1},x - x^{k_j + 1}>
\label{des Omega}
\end{equation}}
From Propositions \ref{prop serie} and \ref{Limitacao}, $\displaystyle\lim_{k\rightarrow\infty}\|x^{k_j} - x^{k_j + 1}\| = 0$ 
and $\|\zeta^{k_j + 1}\|\leq M$ respectively. As $ 0 < l < \mu^k < L, \forall \ k\in N $,
using (\ref{prop-q}) and inequality of Cauchy-Swartz, $\mid \mu^{k_j}q(x^{k_j + 1},x^{k_j})<\zeta^{k_j + 1},x - x^{k_j + 1}> \mid \rightarrow 0$ when $j\rightarrow \infty$. 
Therefore from (\ref{des Omega}), 
\begin{equation}
f(x,z^*) \geq f(x^*,z^*), \ \forall \ x\in \Omega. 
\label{minimo omega}
\end{equation} 
We will now show that $x^*\in {\rm argmin}_w\{F(x) / x\in R^n\}$. Suppose, by contradiction, that there is $\bar{x}\in R^n$ such that $F(\bar{x})\ll F(x^*)$. As 
$z^*\in R^m_+$, for (P3), 
\begin{equation}
f(\bar{x},z^*) < f(x^*,z^*).
\label{x barra}
\end{equation}
As $\Omega^{k+1}\subseteq \Omega^k, \ \forall k\geq 0$ and $x^{k_j}\in\Omega^{k_j - 1}, \ \forall j$ with $x^{k_j}\rightarrow x^*; \ j\rightarrow\infty$, then $x^*\in \Omega$, i.e., $F(x^*)\leq F(x^k), \ \forall k\in N$. Hence, $F(\bar{x}) \ll F(x^k), \ \forall k\in N$, i.e, $\bar{x}\in\Omega$, which contradicts (\ref{minimo omega}) and (\ref{x barra}).
\end{proof}
\section{Regularization with the quasi-distance, q}
\label{caso-q}
In this section it will be shown that, using the $(LQDPS)$ method with the quasi-distance $q$ as regularization in place of $q^2$, the result as to the existence of the iterations remains valid and the sequence generated by the algorithm is limited.
To ensure convergence, we assume that the sequence of parameters $\{\mu^k\} $ satisfies: $\mu^k > 0 \ \forall k = 1,2,...$ and $\mu^k \rightarrow 0$ when $k\rightarrow\infty$, thus confirming the importance of regularization with the quasi-distance $q^2$. 
A different stop criterion will be used.\\
\ \\
{\bf Well-posedness:} Let $\bar{x}\in R^n$ fixed. As $q(.,\bar{x}):R^n\rightarrow R$ is continuous and coercive (see Propositions \ref{prop-qd 1} and \ref{prop-qd 2}), the proposition \ref{boa definicao} is still valid if we replace $q^2$ by $q$ in (\ref{exponencial}). \\
\ \\
{\bf Characterization:} The item (i) in the corollary \ref{Caract} will be replaced by:
\begin{enumerate}
\item[{\rm (i-1)}] {\rm There are} $\xi^{k+1}\in\partial f(.,z^{k+1})(x^{k+1})$, \ \ \ $\zeta^{k+1}\in\partial(q(.,x^k))(x^{k+1})$ \ \ {\rm and} \ \ \\ $v^{k+1}\in N_{\Omega^k}(x^{k+1})$ {\rm such that}
\begin{equation}
 \xi^{k+1} = - \frac{\mu^k}{2}\zeta^{k+1} - v^{k+1}. \label{caract. em x 1}
\end{equation}
\end{enumerate}
{\bf Proof:} As, from Proposition \ref{prop-qd 1}, $\frac{\mu^k}{2}q(.,x^k)$ is Lipschitz continuous in $x^{k+1}$, the proof of (i-1) is similar to the proof of
(i) in the corollary \ref{Caract}.\\ 
The item (ii) in the corollary \ref{Caract} will be the same.\\
\ \\
{\bf Stop Criterion:} We will use the following stop criterion:
\begin{proposicao}[Stop Criterion - Case q]
Let $\{(x^k,z^k)\}_{k\in N}$ be the sequence generated by the $(LQDPS)$ method with $q$ instead of $q^2$. If there is $\tilde{z}\in R^m_+$ such that 
\begin{center}
 $x^{k+1}\in {\rm argmin}\{f(.,\tilde{z}); \ \ x\in \Omega^k\}$ for any integer $k \geq 0$,
\end{center}
then $x^{k+1}$ is a weak pareto solution for the unconstrained multi-objective optimization problem (\ref{probmultiobjetivo}). 
\end{proposicao}
\begin{proof}
As $x^{k+1}\in {\rm argmin}\{f(.,\tilde{z}); \ \ x\in \Omega^k\}$, then 
\begin{equation}
f(x^{k+1},\tilde{z})\leq f(x,\tilde{z}) \ \forall x\in \Omega^k. 
\label{min omegak}
\end{equation}
Suppose, for contradiction, that there is $\bar{x}\in R^n$ such that $F(\bar{x}) \ll F(x^{k+1})$. Then, $\bar{x}\in\Omega^k$ ($\Omega^{k+1} \subseteq \Omega^k \ \forall k\geq 0$ ) and for (P3),
$f(\bar{x},\tilde{z}) < f(x^{k+1},\tilde{z})$, which contradicts (\ref{min omegak}). So $x^{k+1}\in {\rm argmin}_w \{ F(x) / x\in R^n \}$.
\end{proof} \\
{\bf Convergence:} Let $\{(x^k,z^k)\}_{k\in N}$ be a sequence generated by {\bf $(LQDPS)$ Method} with $q$ instead of $q^2$. 
Proposition \ref{prop sequencias} is still valid if we replace $q^2$ by $q$ in (\ref{exponencial}). In fact, clearly Proposition \ref{prop sequencias}(i) and Proposition \ref{prop sequencias}(ii) are still valid. As $q(x^k,x^k) = 0$ and $q(x,y)\geq 0 \ \forall (x,y)\in R^n \times R^n$, if we replace $q^2$ by $q$ in (\ref{exponencial}) the proof of (iii) is similar to the case $q^2$. 
Then {\bf (i)'} $\{x^k\}_{k\in N}$ is bounded; {\bf (ii)'} $\{z^k\}_{k\in N}$ is convergent; {\bf (iii)'} $\{f(x^k,z^k)\}_{k\in N}$ is non-increasing and convergent.\\
Suppose that the sequence of parameters $\{\mu^k\} $ satisfies: $\mu^k > 0 \ \forall k = 1,2,...$ and $\mu^k \rightarrow 0$ when $k\rightarrow\infty$.
We will show that the Theorem \ref{convergencia} (convergence) is still valid if we replace $q^2$ by $q$ in (\ref{exponencial}). Let $\{x^{k_j}\}_{j\in N}$ be a subsequence of $\{x^k\}_{k\in N}$ that satisfies $\displaystyle\lim_{j\rightarrow\infty}x^{k_j} =x^*$ e $z^*\in R^m_+$ such that $\displaystyle\lim_{j\rightarrow\infty}z^{k_j} =z^*$. 
It is easy to see that, if (\ref{caract. em x 1}) is true, the inequality (\ref{des Omega}) becomes
{\small
\begin{equation}
f(x,z^{k_j + 1}) \geq  f(x^{k_j + 1},z^{k_j + 1}) - \frac{\mu^{k}}{2}<\zeta^{k_j + 1},x - x^{k_j + 1}>, \ \forall x\in \Omega.
\label{des Omega 1}
\end{equation}}
From (P2) and (P4), $f$ is continuous in $R^n\times R^m_+$, so that 
$f(x^*,z^*) = \displaystyle\lim_{k\rightarrow\infty}f(x^{k_j},z^{k_j}) =  \displaystyle\inf_{k\in N}\{ f(x^k,z^k)\}$. 
Therefore from (iii)' $\displaystyle\lim_{k\rightarrow\infty}f(x^{k_j + 1},z^{k_j + 1}) = f(x^*,z^*)$.
As $\|\zeta^{k}\| \leq M$ (see Prop. \ref{Limitacao}), $\lim_{j\rightarrow\infty}\mu^{k} = 0$ and $\|x - x^{k +1}\|\leq M_1$,
using Cauchy-Schwarz inequality, we conclude that $\mid \frac{\mu^{k}}{2}<\zeta^{k + 1},x - x^{k + 1}> \mid \rightarrow 0$ when $k\rightarrow \infty$. 
Therefore, from (\ref{des Omega 1}), $f(x,z^*) \geq f(x^*,z^*), \ \forall x \in \Omega^{k}$, and then, similarly to the end of the demonstration of 
Theorem \ref{convergencia}, we conclude that $x^*\in {\rm argmin}_w \{ F(x) / x \in R^n \}$.

\section{Numerical examples}
\label{matlab}
In this section we will implement the $(LQDPS)$ method given in section \ref{metodo}. All numerical experiments were performed using an Intel(R) Core(TM) 2 Duo with Windows 7. The source code is written in Matlab. We tested our method taking into account three multi-objective test functions presented by Li and Zhang in \cite{Hui}, that is, we took into account the following functions: \\ 
\ \\ 
{\bf (a) (\cite{Hui}, function F1, pg. 287):} $F_a = (F_a^1,F_a^2): R^3 \rightarrow R^2$ given by $F_a^1 = x_1 + 2(x_3 - x_1^2)^2$, 
$F_a^2 = 1 - \sqrt{x_1} + 2(x_2 - x_1^{0,5})^2$ and $x_i \in [0,1], \ i = 1,\ 2 ,\ 3$ with the set of all Pareto optimal points (PS) given by
$x_2 = x_1^{0,5}$ and $x_3 = x_1^2$, $ x_1\in [0,1]$. \\
\ \\ 
{\bf (b) (\cite{Hui}, function F4, pg. 287):} $F_b = (F_b^1,F_b^2): R^3 \rightarrow R^2$ 
given by $F_b^1 = x_ 1 + 2 (x_3 - 0,8 x_1 {\rm cos}((6\pi x_1 + \pi)/3))^2$, $F_b^2 = 1 - \sqrt{x_ 1} + 2 ( x_2 - 0,8 x_1 {\rm sin} (6\pi x_1 + 2\pi /3))^2$ and 
$(x_1,x_2,x_3)\in [0,1]\times [-1,1]\times [-1,1]$ with the set (PS) given by $x_2 = 0,8 x_1 {\rm sin} (6\pi x_1 + 2\pi /3)$ and 
$x_3 = 0.8 x_1 {\rm cos} ((6\pi x_1 + \pi)/3)$, $ x_1\in [0,1]$. \\
\ \\
{\bf (c) (\cite{Hui}, function F6, pg. 287):} $F_c = (F_c^1, F_c^2, F_c^3): R^3 \rightarrow R^3$ given by: $F_c^1 = {\rm cos}(0,5 x_1 \pi){\rm cos}(0,5 x_2 \pi)$, 
$F_c^2 = {\rm cos}(0,5 x_1 \pi){\rm sin}(0,5 x_2 \pi)$, $F_c^3 = {\rm sin}(0,5 x_1 \pi) + 2(x_3 - 2 x_2 {\rm sin} (2\pi x_1 + \pi))^2$ 
and $(x_1,x_2,x_3)\in [0,1]\times [0,1] \times [-2,2]$ with the set (PS) given by $x_3 = 2 x_ 2{\rm sin} (2\pi x_1 + \pi)$, $(x_1,x_2)\in [0,1]\times [0,1].$ \\
\ \\
The tests will be performed using the scalarization function proposed in this study (see prop. \ref{exemplo p5}) and the scalarization function proposed by Greg\'orio and Oliveira in \cite{Gregorio} (see function given by (\ref{exemplo Gregorio e Oliveira})). All tests will consider the quasi-distance application $q: R^n \times R^n \rightarrow R_+$ presented by Moreno et al. in \cite{Moreno}; more specifically, we will consider $q: R^3 \times R^3 \rightarrow R_+$ given by $q(x,y) = \sum\limits_{i=1}^3 q_i(x_i,y_i)$ where $q_i(x_i,y_i)= 3(y_i-x_i)$ se $y_i - x_i > 0$ or $q_i(x_i,y_i) = 2(x_i-y_i)$ if  $y_i - x_i \leq 0$.\\
\ \\ 
In the tables, $tol$ denotes the stop criterion tolerance ($\|x^k - x^{k+1}\|_{\infty}\leq tol$); $\mu_k$, $\beta_k$ are the parameters of the $(LQDPS)$ method; $k_i^*, \ i = 1,2$ the number of iterations of the algorithm using the scalarization function $f_i: R^n\times R^m_+\rightarrow R, \ i = 1,2$ where $f_1$ is given by Proposition \ref{exemplo p5} and $f_2$ is given by (\ref{exemplo Gregorio e Oliveira}) and $\|x^*_{k_i^*} - x^*\|_\infty$ is the distance between the approximate solution from $f_i$ and the exact solution, i.e., the error committed with the scalarization function $f_i$. The maximum number of iterations is 100.   
\begin{exem}
In this example, we will consider the multi-objective function $F_a: R^3 \rightarrow R^2$ given above, and the initial iterations $x_0 = (0.5, 0.5, 0.5)$ \ $\in R^3$ and $z_0 = (1, 1)\in R^2_{++}.$ The numeric results are shown in the table below. \\
\ \\
{\small
\colorbox[gray]{0.9}{
\begin{tabular}{l l l l l l l l}
 \hline
No. &  $tol$      &  $\mu_k$    & $\beta_k$ & $k_1^*$ & $\|x^*_{k_1^*} - x^*\|_\infty$ &  $k_2^*$ & $\|x^*_{k_2^*} - x^*\|_\infty$ \\ \hline
1   &  $10^{-2}$  &  $1 + 1/k$  & $1 + 1/k$ &   9     &    5.545339e-003        &     10   &    5.118762e-002   \\
2   &  $10^{-3}$  &  $1 + 1/k$  & $1 + 1/k$ &   28    &    6.247045e-009        &     23   &    6.979995e-003  \\
3   &  $10^{-4}$  &  $1 + 1/k$  & $1 + 1/k$ &   87    &    7.960987e-009        &     62   &    8.279647e-009  \\ \hline
4   &  $10^{-2}$  &  $1 + 1/k$  &   $k$     &    7    &    1.701151e-002        &     9    &    5.726488e-002   \\              
5   &  $10^{-3}$  &  $1 + 1/k$  &   $k$     &   28    &    7.351215e-009        &     24   &    6.281176e-003   \\
6   &  $10^{-4}$  &  $1 + 1/k$  &   $k$     &   100   &    3.576296e-009        &     41   &    8.260435e-009    \\ \hline
7   &  $10^{-2}$  &  $2 - 1/k$  &   $1/k$   &    7    &    2.273775e-002        &     8    &    9.389888e-002   \\
8   &  $10^{-3}$  &  $2 - 1/k$  &   $1/k$   &    15   &    2.790977e-003        &     32   &    1.040779e-002   \\
9   &  $10^{-4}$  &  $2 - 1/k$  &   $1/k$   &    28   &    1.071720e-008        &     100  &    9.213105e-009   \\ \hline
10  &  $10^{-2}$  &  $2 - 1/k$  &   $k$     &    7    &    1.674806e-002        &      8   &    9.413130e-002   \\
11  &  $10^{-3}$  &  $2 - 1/k$  &   $k$     &    27   &    8.168611e-009        &     32   &    1.039107e-002   \\
12  &  $10^{-4}$  &  $2 - 1/k$  &   $k$     &    100  &    8.096220e-009        &     65   &    7.790086e-009   \\ \hline
13  &  $10^{-2}$  &    $1$      &   $1$     &    8    &    6.966285e-003        &     9    &    5.000950e-002   \\
14  &  $10^{-3}$  &    $1$      &   $1$     &    26   &    1.906054e-009        &     23   &    6.138829e-003   \\
15  &  $10^{-4}$  &    $1$      &   $1$     &    83   &    8.254353e-009        &     39   &    1.546241e-005  \\ \hline
\end{tabular}} }
\end{exem}

\begin{exem}
In this example we consider the multi-objective function $F_b: R^3 \rightarrow R^2$ given above, and the initial iterations $x_0 = (0.5, 0.5, 0.5)\in R^3$ and $z_0 = (1, 1)\in R^2_{++}.$ the numeric results are shown in the table below. \\
\ \\ 
{\small
\colorbox[gray]{0.9}{
\begin{tabular}{l l l l l l l l}
 \hline
No. &  $tol$      &  $\mu_k$    & $\beta_k$ & $k_1^*$ & $\|x^*_{k_1^*} - x^*\|$ &  $k_2^*$ & $\|x^*_{k_2^*} - x^*\|_\infty$ \\ \hline
1   &  $10^{-2}$  &  $1 + 1/k$  & $1 + 1/k$ &    10   &     4.419117e-003       &    10    &    3.800596e-002   \\
2   &  $10^{-3}$  &  $1 + 1/k$  & $1 + 1/k$ &    29   &     7.617346e-009       &    20    &    5.872760e-003   \\
3   &  $10^{-4}$  &  $1 + 1/k$  & $1 + 1/k$ &    92   &     7.831306e-009       &    100   &    8.102059e-009   \\ \hline
4   &  $10^{-2}$  &  $1 + 1/k$  &   $k$     &    7    &     1.423293e-002       &    10    &    3.771631e-002   \\              
5   &  $10^{-3}$  &  $1 + 1/k$  &   $k$     &    20   &     4.560126e-009       &    21    &    5.533943e-003   \\
6   &  $10^{-4}$  &  $1 + 1/k$  &   $k$     &    98   &     6.872232e-009       &    38    &    1.000619e-007   \\ \hline
7   &  $10^{-2}$  &  $2 - 1/k$  &   $1/k$   &    7    &     2.265857e-002       &    9     &    6.038495e-002   \\
8   &  $10^{-3}$  &  $2 - 1/k$  &   $1/k$   &    15   &     3.304754e-003       &    25    &    8.176106e-003   \\
9   &  $10^{-4}$  &  $2 - 1/k$  &   $1/k$   &    28   &     7.814512e-009       &    100   &    7.497307e-009   \\ \hline
10  &  $10^{-2}$  &  $2 - 1/k$  &   $k$     &    7    &     1.251182e-002       &    7     &    6.525365e-002   \\
11  &  $10^{-3}$  &  $2 - 1/k$  &   $k$     &    30   &     8.735791e-009       &    23    &    8.117802e-003   \\
12  &  $10^{-4}$  &  $2 - 1/k$  &   $k$     &    100  &     5.561728e-009       &    52    &    7.547563e-009   \\ \hline
13  &  $10^{-2}$  &    $1$      &   $1$     &    8    &     5.099261e-003       &    10    &    4.231940e-002   \\
14  &  $10^{-3}$  &    $1$      &   $1$     &    27   &     5.045036e-009       &    22    &    5.051759e-003   \\
15  &  $10^{-4}$  &    $1$      &   $1$     &    88   &     8.499673e-009       &    82    &    9.235203e-010                \\ \hline
\end{tabular}} }
\end{exem}
\begin{exem}
In this example we consider the multi-objective function $F_c: R^3 \rightarrow R^3$ given above, and the initial iterations $x_0 = (0.5, 0.5, 0.5)\in R^3$ and $z_0 = (1, 1, 1)\in R^3_{++}.$ The numeric results are shown in the table below.\\
\ \\ 
{\small
\colorbox[gray]{0.9}{
\begin{tabular}{l l l l l l l l}
 \hline
No. &  $tol$      &  $\mu_k$    & $\beta_k$ & $k_1^*$ & $\|x^*_{k_1^*} - x^*\|$ &  $k_2^*$ & $\|x^*_{k_2^*} - x^*\|_\infty$ \\ \hline
1   &  $10^{-2}$  &  $1 + 1/k$  & $1 + 1/k$ &   10    &     1.066481e-002       &    18    &   5.068830e-002     \\
2   &  $10^{-3}$  &  $1 + 1/k$  & $1 + 1/k$ &   31    &     1.698174e-008       &    33    &   5.315241e-003   \\
3   &  $10^{-4}$  &  $1 + 1/k$  & $1 + 1/k$ &   100   &     5.432795e-009       &    100   &   1.130028e-008    \\ \hline
4   &  $10^{-2}$  &  $1 + 1/k$  &   $k$     &   10    &     9.733382e-003       &    19    &   5.908485e-002    \\              
5   &  $10^{-3}$  &  $1 + 1/k$  &   $k$     &   28    &     4.176586e-010       &    34    &   2.307318e-007    \\
6   &  $10^{-4}$  &  $1 + 1/k$  &   $k$     &   100   &     7.086278e-011       &    35    &   7.450585e-009    \\ \hline
7   &  $10^{-2}$  &  $2 - 1/k$  &   $1/k$   &   11    &     2.653977e-002       &    20    &   9.899806e-002    \\
8   &  $10^{-3}$  &  $2 - 1/k$  &   $1/k$   &   18    &     2.046561e-007       &    47    &   1.059293e-002    \\
9   &  $10^{-4}$  &  $2 - 1/k$  &   $1/k$   &   33    &     1.161832e-008       &    100   &   9.253656e-009    \\ \hline
10  &  $10^{-2}$  &  $2 - 1/k$  &   $k$     &   11    &     1.835990e-002       &    22    &   8.862843e-002   \\
11  &  $10^{-3}$  &  $2 - 1/k$  &   $k$     &   28    &     5.441347e-010       &    48    &   1.047200e-002   \\
12  &  $10^{-4}$  &  $2 - 1/k$  &   $k$     &   100   &     9.476497e-010       &    75    &   2.793537e-009    \\ \hline
13  &  $10^{-2}$  &    $1$      &   $1$     &    9    &     8.326796e-003       &    17    &   5.182457e-002   \\
14  &  $10^{-3}$  &    $1$      &   $1$     &    29   &     1.343175e-008       &    32    &   4.799238e-003   \\
15  &  $10^{-4}$  &    $1$      &   $1$     &    96   &     6.882171e-009       &   100    &   3.961009e-009    \\ \hline
\end{tabular}} }
\end{exem}


\section{Conclusions}
\paragraph{}
Greg\'orio and Oliveira \cite{Gregorio} presented an example of a function that satisfies the properties (P1) to (P4). 
Based on Fliege and Svaiter \cite{Fliege}, Greg\'orio and Oliveira assumed that $\Omega^0$ is bounded, and they established 
the convergence of the logarithm-quadratic proximal scalarization method.

In this study, we propose adding a condition to one of the objective functions, which limits $\Omega^0$. We also 
propose another example of a function that satisfies the properties (P1) to (P4). As a variation of the 
logarithm-quadratic proximal scalarization method of Greg\'orio and Oliveira \cite{Gregorio}, we replaced the 
quadratic term with a quasi-distance, which entails losing important properties, such as convexity and differentiability. 
Proceeding differently, however, the convergence of the method was proved. Finally, some numerical examples of the $(LQDPS)$ method were presented.\\
\ \\
{\bf Acknowledgments.} The authors are grateful to Dr. Michael Sousa (UFC-Brazil) for its aid in the implementation of the (LQDPS) Method.


\end{document}